\documentclass[12pt,twoside]{amsart}
\usepackage{amssymb,amsmath,amsthm, amscd, enumerate, mathrsfs}
\usepackage{graphicx, hhline}
\usepackage[all]{xy}
\usepackage{hyperref}
\hypersetup{colorlinks=true}

\title[On subadditivity of the logarithmic Kodaira dimension]
{On subadditivity of the logarithmic Kodaira dimension}
\author{Osamu Fujino} 
\date{2016/2/1, version 0.30}
\address{Department of Mathematics, Graduate School of Science, 
Kyoto University, Kyoto 606-8502, Japan}
\email{fujino@math.kyoto-u.ac.jp}

\newcommand{\rank}[0]{{\operatorname{rank}}}
\newcommand{\Supp}[0]{{\operatorname{Supp}}}
\newtheorem{thm}{Theorem}[section]
\newtheorem{lem}[thm]{Lemma}
\newtheorem{cor}[thm]{Corollary}
\newtheorem{prop}[thm]{Proposition}
\newtheorem{conj}[thm]{Conjecture}

\theoremstyle{definition}
\newtheorem{defn}[thm]{Definition}
\newtheorem{rem}[thm]{Remark}
\newtheorem{ex}[thm]{Example}
\newtheorem*{ack}{Acknowledgments} 
\newtheorem{say}[thm]{}

\makeatletter
    
    \@addtoreset{equation}{section}
\makeatother
\begin{document}

\maketitle 

\begin{abstract}
We reduce Iitaka's subadditivity 
conjecture for the logarithmic Kodaira dimension 
to a special case of the generalized abundance conjecture 
by establishing an Iitaka type inequality for Nakayama's 
numerical Kodaira dimension. 
Our proof heavily depends on Nakayama's 
theory of $\omega$-sheaves and 
$\widehat{\omega}$-sheaves. 
As an application, we prove the subadditivity of the logarithmic 
Kodaira dimension for affine varieties by using 
the minimal model program for projective klt pairs with 
big boundary divisor. 
\end{abstract}

\tableofcontents 

\section{Introduction}

In this paper, we discuss Iitaka's subadditivity conjecture on the logarithmic 
Kodaira dimension $\overline \kappa$. 

\begin{conj}[Subadditivity of logarithmic Kodaira dimension]\label{f-conj1.1}
Let $g:V\to W$ be a dominant morphism between algebraic varieties. 
Then we have the following inequality 
$$
\overline \kappa (V)\geq \overline\kappa (F')+\overline \kappa (W)
$$ 
where $F'$ is an irreducible component of a sufficiently general 
fiber of $g:V\to W$. 
\end{conj}

Conjecture \ref{f-conj1.1} is usually called 
Conjecture $\overline{C}_{n, m}$ 
when $\dim V=n$ and $\dim W=m$. 
If $V$ is complete in Conjecture \ref{f-conj1.1}, then 
it is nothing but the famous Iitaka subadditivity conjecture for 
the Kodaira dimension $\kappa$. 
We see that Conjecture \ref{f-conj1.1} is equivalent to: 

\begin{conj}\label{f-conj1.2}
Let $f:X\to Y$ be a surjective morphism between smooth projective 
varieties with connected fibers. 
Let $D_X$ $($resp.~$D_Y$$)$ be a simple normal crossing 
divisor on $X$ $($resp.~$Y$$)$. 
Assume that $\Supp f^*D_Y\subset \Supp D_X$. 
Then we have 
$$
\kappa(X, K_X+D_X)\geq \kappa
(F, K_F+D_X|_F)+\kappa(Y, K_Y+D_Y)
$$
where $F$ is a sufficiently general fiber of $f:X\to Y$. 
\end{conj}

One of the main purposes of this paper is to prove: 

\begin{thm}[Main theorem]\label{f-thm1.3}
Let $f:X\to Y$ be a surjective morphism between smooth projective 
varieties with connected fibers. 
Let $D_X$ $($resp.~$D_Y$$)$ be a simple normal crossing 
divisor on $X$ $($resp.~$Y$$)$. 
Assume that $\Supp f^*D_Y\subset \Supp D_X$. 
Then we have 
$$
\kappa _{\sigma}(X, K_X+D_X)\geq \kappa _{\sigma} 
(F, K_F+D_X|_F)+\kappa_{\sigma}(Y, K_Y+D_Y)
$$
where $F$ is a sufficiently general fiber of $f:X\to Y$. 
\end{thm}

Note that $\kappa_{\sigma}$ denotes Nakayama's numerical 
Kodaira dimension and that the inequality 
$\kappa_{\sigma}\geq \kappa$ always holds, where 
$\kappa$ is Iitaka's $D$-dimension. 
Theorem \ref{f-thm1.3} is a variant of 
Nakayama's theorem (see 
\cite[V.4.1.~Theorem]{nakayama} and Remark \ref{f-rem3.8}). 
By Theorem \ref{f-thm1.3}, Conjecture \ref{f-conj1.2} is reduced to: 

\begin{conj}\label{f-conj1.4} 
Let $X$ be a smooth projective variety and let $D_X$ be 
a simple normal crossing divisor on $X$. 
Then the equality 
$$
\kappa _\sigma(X, K_X+D_X)=\kappa (X, K_X+D_X) 
$$
holds. 
\end{conj}

Conjecture \ref{f-conj1.4} is known as 
a special case of the generalized abundance 
conjecture (see Conjecture \ref{f-conj2.10}), which is one of the 
most important conjectures for higher-dimensional algebraic varieties. 
As an easy corollary of Theorem \ref{f-thm1.3}, we have: 

\begin{cor}\label{f-cor1.5} 
In Theorem \ref{f-thm1.3}, 
we further assume that $\dim X\leq 3$. 
Then we have 
\begin{align*}
\kappa (X, K_X+D_X)&=\kappa_\sigma(X, K_X+D_X)\\
&\geq \kappa _\sigma(F, K_F+D_X|_F)+\kappa_\sigma(Y, K_Y+D_Y)\\
&\geq \kappa (F, K_F+D_X|_F)+\kappa (Y, K_Y+D_Y). 
\end{align*}
In particular, if $g:V\to W$ is a dominant morphism between algebraic 
varieties with $\dim V\leq 3$, then we have the inequality 
$$
\overline \kappa (V)\geq \overline \kappa (F')+\overline \kappa (W)
$$
where $F'$ is an irreducible component of a sufficiently general fiber of 
$g:V\to W$. 
\end{cor}
Note that the 
equality 
$\kappa _\sigma(X, K_X+D_X)=\kappa (X, K_X+D_X)$ in 
Corollary \ref{f-cor1.5} follows 
from the minimal model program and 
the abundance theorem for $(X, D_X)$ (see Proposition \ref{f-prop4.3}). 
We also note that Corollary \ref{f-cor1.5} is 
new when $\dim V=3$ and $\dim W=1$. 
Anyway, Conjecture \ref{f-conj1.1} now becomes a consequence of 
the minimal model program and the abundance conjecture by Theorem 
\ref{f-thm1.3} (see Remark \ref{f-rem4.5}). This fact strongly 
supports Conjecture \ref{f-conj1.1}. 

As an application of Theorem \ref{f-thm1.3}, we obtain: 

\begin{cor}[Subadditivity of the logarithmic Kodaira dimension 
for affine varieties]\label{f-cor1.6}
Let $g:V\to W$ be a dominant morphism from an affine variety $V$. 
Then we have 
$$\overline \kappa(V)\geq \overline\kappa (F')+\overline \kappa(W)
$$
where $F'$ is an irreducible component of 
a sufficiently general fiber of $g:V\to W$. 
\end{cor}

Note that $W$ is not necessarily assumed to be 
affine in Corollary \ref{f-cor1.6}. 
In order to prove Corollary \ref{f-cor1.6}, we construct $(X, D_X)$ 
with $\overline \kappa (V)=\kappa (X, K_X+D_X)$ such that 
$(X, D_X)$ has a good minimal model 
or a Mori fiber space structure by using the minimal model 
program for projective klt pairs with big boundary divisor. 
Note that $\kappa _\sigma(X, K_X+D_X)=\kappa (X, K_X+D_X)$ holds for 
such $(X, D_X)$.  

\begin{rem}\label{f-rem1.7}
By the proof of Corollary \ref{f-cor1.6}, we see that 
the inequality 
$$\overline \kappa (V)\geq \overline\kappa (F')+\overline \kappa (W)$$ holds 
for every {\em{strictly}} rational dominant map 
$g:V\dashrightarrow W$ from an affine variety $V$. 
\end{rem}

In this paper, we 
use Nakayama's theory of $\omega$-sheaves and $\widehat \omega$-sheaves 
in order to prove an Iitaka type inequality for Nakayama's numerical 
Kodaira dimension (see Theorem \ref{f-thm1.3}). 
It is closely related to Viehweg's clever covering trick and weak positivity. 
We also use the minimal model program for 
projective klt pairs with big boundary divisor for the study of affine varieties 
(see Section \ref{f-sec4}). 

\begin{rem}\label{f-rem1.8}
Let $f:X\to Y$ be a projective surjective morphism 
between smooth projective varieties with connected fibers. 
In \cite{kawamata}, Kawamata proved that 
the inequality 
$$\kappa (X)\geq \kappa (F)+\kappa (Y),$$
where $F$ is a sufficiently general fiber of $f:X\to Y$, 
holds under the assumption that 
the geometric generic fiber $X_{\overline \eta}$ of 
$f:X\to Y$ has a good minimal model. 
His approach is completely different from ours. 
For the details, see \cite{kawamata}.  
\end{rem}

Finally the following theorem, which is a 
slight generalization of Theorem \ref{f-thm1.3}, was 
suggested by the referee.  

\begin{thm}\label{f-thm1.9}
Let $f:X\to Y$ be a proper surjective morphism 
from a normal variety $X$ onto a smooth 
complete variety $Y$ with connected fibers. 
Let $D_X$ be an effective $\mathbb Q$-divisor 
on $X$ such that $(X, D_X)$ is lc 
and let $D_Y$ be a simple normal crossing divisor on $Y$. 
Assume that $\Supp f^*D_Y\subset \lfloor D_X\rfloor$, 
where $\lfloor D_X\rfloor$ is the round-down of $D_X$. 
Then we have 
$$
\kappa _\sigma (X, K_X+D_X)\geq 
\kappa _\sigma(F, K_F+D_X|_F)+\kappa _\sigma 
(Y, K_Y+D_Y) 
$$ 
where $F$ is a sufficiently general fiber of $f:X\to Y$. 
\end{thm}

The formulation of Theorem \ref{f-thm1.9} seems to be 
natural and useful from the minimal model theoretic 
viewpoint. 

We summarize the contents of this paper. 
In Section \ref{f-sec2}, we briefly recall Iitaka's logarithmic Kodaira dimension, 
Nakayama's numerical Kodaira dimension, Nakayama's $\omega$-sheaves and 
$\widehat {\omega}$-sheaves, and some related topics. 
In Section \ref{f-sec3}, 
we prove Theorem \ref{f-thm1.3}, which is the main theorem of this paper. 
Our proof heavily depends on Nakayama's argument in his book \cite{nakayama}, 
which is closely related to Viehweg's covering trick and weak positivity. 
In Section \ref{f-sec4}, 
we discuss the minimal model program for affine varieties. 
For any affine variety, we see that there is a smooth compactification which has 
a good minimal model or a Mori fiber space structure. 
As an application, we obtain the subadditivity of the logarithmic Kodaira 
dimension for affine varieties by 
Theorem \ref{f-thm1.3} (see Corollary \ref{f-cor1.6}). 

\begin{ack}
The author was partially supported by Grant-in-Aid for 
Young Scientists (A) 24684002 from JSPS. 
He thanks Professors Yoshinori Gongyo, Takeshi Abe, 
Noboru Nakayama, and the referee for 
useful and helpful comments. He thanks Universit\'e Lille 1 for 
its hospitality. 
He would like to thank Professor 
De-Qi Zhang for pointing out a mistake. 
Finally, he thanks Professors Thomas Eckl and 
Brian Lehmann for answering his questions on various 
numerical dimensions. 
\end{ack}

We will work over $\mathbb C$, the complex number field, throughout this paper. 
For the standard notation of the minimal model program, see \cite{fujino-fund} 
and \cite{fujino-foundation}. 

\section{Preliminaries}\label{f-sec2}

In this section, we quickly explain the logarithmic Kodaira dimension 
introduced by Iitaka, 
Nakayama's numerical Kodaira dimension, 
$\omega$-sheaves, and $\widehat \omega$-sheaves. 

\begin{say}[Sufficiently general fibers] 
Let us recall the definition of sufficiently general fibers for the reader's 
convenience. 

\begin{defn}[Sufficiently general fibers] 
Let $f:X\to Y$ be a morphism between algebraic varieties. 
Then a {\em{sufficiently general fiber}} $F$ of $f:X\to Y$ 
means that $F=f^{-1}(y)$ where $y$ is 
any point contained in a countable intersection of 
nonempty Zariski open 
subsets of $Y$. 
\end{defn}
A sufficiently general fiber is sometimes called a 
{\em{very general fiber}} in the literature. 
\end{say}

\begin{say}[Logarithmic Kodaira dimension]
The notion of the logarithmic Kodaira dimension was introduced by 
Shigeru Iitaka (see \cite{iitaka}). 

\begin{defn}[Logarithmic Kodaira dimension]\label{f-def2.4}
Let $V$ be an irreducible algebraic variety. 
By Nagata's theorem, we have a complete algebraic variety 
$\overline V$ which contains $V$ as a dense Zariski open subset. 
By Hironaka's theorem, we have a smooth projective 
variety $\overline W$ and a projective 
birational morphism $\mu:\overline W\to \overline V$ such that 
if $W=\mu^{-1}(V)$, then $\overline D=\overline W-W=\mu^{-1}(\overline V-V)$ 
is a simple normal crossing divisor on $\overline W$. 
The {\em{logarithmic Kodaira dimension}} $\overline \kappa(V)$ of 
$V$ is defined as 
$$
\overline \kappa(V)=\kappa(\overline W, K_{\overline {W}}+\overline D) 
$$
where $\kappa$ denotes Iitaka's $D$-dimension. 
\end{defn}
It is well-known and is 
easy to see that $\overline \kappa(V)$ is well-defined, that is, 
it is independent of the choice of the pair $(\overline W, \overline D)$. 

As we have already explained, 
the following conjecture (see Conjecture \ref{f-conj1.1}) is usually called Conjecture 
$\overline C_{n, m}$ when 
$\dim V=n$ and $\dim W=m$. 

\begin{conj}[Subadditivity of logarithmic Kodaira dimension]\label{f-conj2.5}
Let $g:V\to W$ be a dominant morphism between algebraic varieties. 
Then we have the following inequality 
$$
\overline \kappa (V)\geq \overline\kappa (F')+\overline \kappa (W)
$$ 
where $F'$ is an irreducible component of a sufficiently general 
fiber of $g:V\to W$. 
\end{conj}

Note that Conjecture \ref{f-conj1.2} is a special case of 
Conjecture \ref{f-conj2.5} by putting $V=X\setminus D_X$ and $W=Y\setminus D_Y$. 
On the other hand, we can easily 
check that Conjecture \ref{f-conj2.5} follows from Conjecture 
\ref{f-conj1.2}. 
For the details, see the proof of Corollary \ref{f-cor1.6}. 
Anyway, Conjecture \ref{f-conj2.5} (see Conjecture \ref{f-conj1.1}) is 
equivalent to Conjecture \ref{f-conj1.2}. 
We note that Conjecture \ref{f-conj1.2} is easier to handle than Conjecture 
\ref{f-conj2.5} from the 
minimal model theoretic viewpoint. 
\end{say}

\begin{say}[Nakayama's numerical Kodaira dimension]
Let us recall the definition of 
Nakayama's numerical Kodaira dimension. 

\begin{defn}[Nakayama's numerical Kodaira dimension]\label{f-def2.7} 
Let $X$ be a smooth projective variety and let $D$ be a Cartier divisor 
on $X$. We put 
$$
\sigma(D; A)=\max\left\{k\in \mathbb Z_{\geq 0}\,\left| 
\,\underset{m\to \infty}{\limsup} \frac{\dim H^0(X, \mathcal O_X
(A+mD))}
{m^k}>0\right.\right\}
$$
and 
$$
\kappa_\sigma(X, D)=\max\{\sigma(D; A)\, |\, {\text{$A$ is a divisor}}\}. 
$$
Note that if $H^0(X, \mathcal O_X(A+mD))\ne 0$ only for 
finitely many 
$m\in \mathbb Z_{\geq 0}$ then we define $\sigma(D; A)=-\infty$. 
It is obvious that $\kappa_\sigma(X, D)\geq \kappa(X, D)$, 
where $\kappa(X, D)$ denotes Iitaka's $D$-dimension of $D$. 
We also note that $\kappa_\sigma(X, D)\geq 0$ if and 
only if $D$ is pseudo-effective (see \cite[V.1.4.~Corollary]{nakayama}).  

When $X$ is a normal projective 
variety, we take a resolution $\varphi: X'\to X$, 
where $X'$ is a smooth projective variety, 
and put 
$$
\kappa_\sigma(X, D)=\kappa_{\sigma}(X', \varphi^*D). 
$$ 
It is not difficult to see that 
$\kappa_\sigma(X, D)$ is well-defined and has various good properties. 
For the details, see \cite[V.\S2]{nakayama}, \cite{lehmann} 
and \cite{eckl}. 
\end{defn}

The following lemma, which is lacking 
in \cite{nakayama}, will play a crucial 
role in the proof of Theorem \ref{f-thm1.3}. 

\begin{lem}[{\cite[Theorem 6.7 (7)]{lehmann}}]\label{f-lem2.8} 
Let $D$ be a pseudo-effective 
Cartier divisor on a smooth 
projective variety $X$. 
We fix some sufficiently ample Cartier divisor 
$A$ on $X$. 
Then there exist positive constants $C_1$ and $C_2$ 
such that 
\begin{equation*}
C_1m^{\kappa_{\sigma}(X, D)}\leq 
\dim H^0(X, \mathcal O_X(mD+A))\leq 
C_2m^{\kappa _{\sigma}(X, D)}
\end{equation*} 
for every sufficiently large $m$. 
\end{lem}

For the details, see \cite[Theorem 6.7 (7)]{lehmann} and 
\cite[2.8, 2.10, and Theorem 0.2]{eckl}. 

\begin{rem}\label{f-rem2.9}
Nakayama's numerical Kodaira dimension can be defined for $\mathbb R$-Cartier 
$\mathbb R$-divisors and has many equivalent definitions and 
several nontrivial characterizations. 
For the details, see \cite[V.\S2]{nakayama}, 
\cite[Theorem 1.1]{lehmann}, and \cite[Theorem 0.2]{eckl}. 
Note that \cite[2.9]{eckl} describes a gap in 
Lehmann's paper \cite{lehmann}. 
\end{rem}

The following conjecture is one of the most important 
conjectures for 
higher-dimensional algebraic varieties. 
Conjecture \ref{f-conj1.4} is a special case of Conjecture \ref{f-conj2.10}. 

\begin{conj}[Generalized abundance conjecture]\label{f-conj2.10}
Let $(X, \Delta)$ be a $\mathbb Q$-factorial projective 
dlt pair. 
Then $\kappa_\sigma(X, K_X+\Delta)=\kappa_\iota(X, K_X+\Delta)$.
\end{conj}

It is obvious that if Conjecture \ref{f-conj2.10} holds for $(X, D_X)$ in 
Theorem \ref{f-thm1.3} then Theorem \ref{f-thm1.3} implies 
Conjecture \ref{f-conj1.2} 
in full generality. 

\begin{rem}[On the definition of $\kappa_\iota (X, K_X+\Delta)$]\label{f-rem2.11}
We have to be careful when $\Delta$ is an $\mathbb R$-divisor 
in Conjecture \ref{f-conj2.10}. 
If there exists an effective $\mathbb R$-divisor 
$D$ on $X$ such that $K_X+\Delta\sim _{\mathbb R}D$, then 
we put 
$$
\kappa_\iota (X, K_X+\Delta)=\underset{m\to \infty}{\limsup}\frac
{\log \dim H^0(X, \mathcal O_X(\lfloor mD\rfloor))}{\log m}. 
$$
Otherwise, we put $\kappa_\iota (X, K_X+\Delta)=-\infty$. 
The above definition of $\kappa_\iota (X, K_X+\Delta)$ is well-defined, 
that is, $\kappa_\iota (X, K_X+\Delta)$ is independent of the choice of $D$ 
(see, for example, \cite[Definition 2.2.1]{choi} and \cite{fujino-foundation}). 
Note that if $K_X+\Delta$ is a $\mathbb Q$-divisor 
then $\kappa _\iota(X, K_X+\Delta)$ coincides 
with $\kappa (X, K_X+\Delta)$, that is, 
$$
\kappa_\iota (X, K_X+\Delta)=
\underset{m\to \infty}{\limsup}\frac
{\log \dim H^0(X, \mathcal O_X(\lfloor m(K_X+\Delta)\rfloor))}{\log m}.  
$$ 
\begin{ex}
We put $X=\mathbb P^1$. 
Let $\Delta$ be an effective $\mathbb R$-divisor on $X$. 
We assume that 
$\deg \Delta=2$ and that 
$\Delta$ is not a $\mathbb Q$-divisor. 
Then we can easily see that 
$K_X+\Delta\sim _{\mathbb R} 0$, 
$\kappa _\sigma (X, K_X+\Delta)=\kappa _{\iota}
(X, K_X+\Delta)=0$, and 
$\kappa (X, K_X+\Delta)=-\infty$. 
\end{ex}

Anyway, we do not use $\mathbb R$-divisors in this paper. 
So, we do not discuss subtle problems on $\mathbb R$-divisors here. 
However, we note that it is indispensable to 
treat $\mathbb R$-divisors when we discuss Conjecture \ref{f-conj2.10} and 
Conjecture \ref{f-conj2.13} below. 
\end{rem}

Note that Conjecture \ref{f-conj2.10} holds in dimension $\leq n$ if and only if 
Conjecture \ref{f-conj2.13} holds in dimension $\leq n$. 

\begin{conj}[Good minimal model conjecture]\label{f-conj2.13} 
Let $(X, \Delta)$ be a $\mathbb Q$-factorial projective 
dlt pair. Assume that $K_X+\Delta$ is pseudo-effective. 
Then $(X, \Delta)$ has a good minimal model. 
\end{conj}

For the relationships among various conjectures on the minimal model 
program, see \cite{fujino-gongyo}. 

We will use the following easy well-known 
lemma in the proof of Corollary \ref{f-cor1.6}. 

\begin{lem}\label{f-lem2.14}
Let $f:X\to Y$ be a generically finite surjective morphism 
between smooth projective varieties. 
Let $D_X$ $($resp.~$D_Y$$)$ be a simple normal crossing divisor on $X$ 
$($resp.~$Y$$)$. 
Assume that $\Supp f^*D_Y\subset \Supp D_X$. 
Then we have 
$$
\kappa(X, K_X+D_X)\geq \kappa (Y, K_Y+D_Y)
$$
and 
$$
\kappa _{\sigma}(X, K_X+D_X)\geq \kappa_\sigma(Y, K_Y+D_Y). 
$$
\end{lem}
\begin{proof}
We put $n=\dim X=\dim Y$. 
Then we have 
$$
f^*\Omega^n_Y(\log D_Y)\subset \Omega^n_X(\log D_X). 
$$ 
Therefore, we can write 
$$
K_X+D_X=f^*(K_Y+D_Y)+R
$$ 
for some effective Cartier divisor $R$. 
Thus, we have the desired inequalities. 
\end{proof}
\end{say}

\begin{say}[Nakayama's $\omega$-sheaves and $\widehat\omega$-sheaves] 

Let us briefly recall the theory of 
Nakayama's $\omega$-sheaves and $\widehat{\omega}$-sheaves. 

The following definition of $\omega$-sheaf is equivalent to 
Nakayama's original definition of $\omega$-sheaf in the category of projective 
varieties (see \cite[V.3.8.~Definition]{nakayama}). 

\begin{defn}[$\omega$-sheaf]
A coherent sheaf $\mathcal F$ on a projective variety $Y$
is called an {\em{$\omega$-sheaf}} if there exists a 
projective morphism $f:X\to Y$ from a smooth 
projective variety $X$ such that $\mathcal F$ is a direct summand of 
$f_*\omega_X$. 
\end{defn}

We also need the notion of $\widehat\omega$-sheaf (see 
\cite[V.3.16.~Definition]{nakayama}). 

\begin{defn}[$\widehat{\omega}$-sheaf]
A coherent torsion-free sheaf $\mathcal F$ on a normal projective variety $Y$ 
is called an {\em{$\widehat{\omega}$-sheaf}} 
if there exist an $\omega$-sheaf $\mathcal G$ and a generically 
isomorphic inclusion $\mathcal G\hookrightarrow \mathcal F^{**}$ into the double 
dual $\mathcal F^{**}$ of $\mathcal F$. 
\end{defn}

Although the following lemma is easy to prove, 
it plays a crucial role in the proof of Theorem \ref{f-thm1.3}. 

\begin{lem}\label{f-lem2.18}
Let $Y$ be a projective variety. 
Then there exists an ample Cartier divisor $A$ on $Y$ such that 
$\mathcal F\otimes \mathcal O_Y(A)$ is generated 
by global sections for every $\omega$-sheaf $\mathcal F$ on $Y$. 
\end{lem}

\begin{proof}
We may assume that 
$\mathcal F=f_*\omega_X$ for a projective morphism 
$f:X\to Y$ from a smooth projective variety $X$. 
Let $H$ be an ample Cartier divisor on $Y$ such that $|H|$ is free. 
We put $A=(\dim Y+1)H$. 
Then we have 
$$
H^i(Y, \mathcal F\otimes \mathcal O_Y(A)\otimes \mathcal 
O_Y(-iH))=0 
$$ 
for every $i>0$ by Koll\'ar's vanishing theorem. 
Therefore, by using the Castelnuovo--Mumford regularity, 
we see that 
$\mathcal F\otimes \mathcal O_Y(A)$ is generated by global sections. 
\end{proof}

As an obvious corollary of Lemma \ref{f-lem2.18}, we have: 

\begin{cor}\label{f-cor2.19}
Let $Y$ be a normal projective variety. 
Then there exists an ample Cartier divisor $A$ on $Y$ such that 
$\mathcal F\otimes \mathcal O_Y(A)$ is generically generated by 
global sections for every reflexive $\widehat{\omega}$-sheaf $\mathcal F$ on $Y$. 
\end{cor}
\end{say}

\begin{say}[Strictly rational map]
We close this section with the notion of strictly rational maps. 
For the details, see \cite[Lecture 2]{iitaka2} and 
\cite[\S2.12 Strictly Rational Maps]{iitaka3}. 

\begin{defn}[Strictly rational map]
Let $f:X\dashrightarrow Y$ be 
a rational map between irreducible 
varieties. 
If there is a proper birational morphism $\mu:Z\to X$ from an irreducible 
variety $Z$ such that 
$f\circ \mu$ is a morphism, then $f:X\dashrightarrow Y$ is called a 
{\em{strictly rational map}}. 
$$
\xymatrix{
Z\ar[d]_{\mu}\ar[dr]^{f\circ \mu}&\\
X\ar@{-->}[r]_{f}&Y
}
$$
\end{defn} 
Note that a rational map $f:X\dashrightarrow Y$ from $X$ to 
a complete variety $Y$ is always strictly rational. 

\begin{ex}
Let $X$ be a smooth 
projective variety and let $U$ be a dense open subset of $X$ 
such that $U\subsetneq X$. 
Then the natural open immersion $\iota: U\hookrightarrow X$ 
is strictly rational. On the other hand, 
$f=\iota^{-1}: X\dashrightarrow U$ is not strictly rational. 
\end{ex}
\end{say}

\section{Subadditivity of Nakayama's numerical Kodaira dimension}\label{f-sec3}

In this section, we prove Theorem \ref{f-thm1.3} by using Nakayama's theory 
of $\omega$-sheaves and $\widehat\omega$-sheaves. 
The following lemma is a special case of \cite[V.3.34.~Lemma]{nakayama}. 
It is a reformulation and a generalization of Viehweg's deep result 
(see \cite[Corollary 5.2]{viehweg}). 

\begin{lem}[{cf.~\cite[V.3.34.~Lemma]{nakayama}}]\label{f-lem3.1}
Let $f:X\to Y$ be a projective surjective morphism 
from a normal projective variety $X$ onto a smooth 
projective variety $Y$ with connected fibers. 
Let $L$ be a Cartier divisor on $X$, let $\Delta$ be an effective 
$\mathbb Q$-divisor on $X$, 
and let $k$ be an integer greater than one satisfying 
the following 
conditions: 
\begin{itemize}
\item[(i)] $(X, \Delta)$ is klt. 
\item[(ii)] $L-k(K_{X/Y}+\Delta)$ is ample. 
\end{itemize}
Then we obtain that 
$$
\omega_Y((k-1)H)\otimes f_*\mathcal O_X(L)
$$ 
is an $\widehat \omega$-sheaf for any 
ample Cartier divisor $H$ on $Y$. 
\end{lem}

\begin{rem}\label{f-rem3.2}In Lemma \ref{f-lem3.1}, 
it is sufficient to assume that $(X, \Delta)$ is lc and that 
there is a positive rational number $\delta$ such that 
$(X, (1-\delta)\Delta)$ is klt. 
This is because $L-k(K_{X/Y}+(1-\varepsilon)\Delta)$ 
is ample and $(X, (1-\varepsilon)\Delta)$ is klt for 
$0<\varepsilon \ll \delta$. 
Therefore, we can replace $(X, \Delta)$ with 
$(X, (1-\varepsilon)\Delta)$ and may assume that 
$(X, \Delta)$ is klt. 
\end{rem}

We do not repeat the proof of \cite[V.3.34.~Lemma]{nakayama} here. 
For the details, see \cite{nakayama}. 
Note that the essence of Viehweg's theory of weakly positive sheaves 
is contained in the proof of Lemma \ref{f-lem3.1}. 
Therefore, Lemma \ref{f-lem3.1} is highly nontrivial. 

We make a small remark on the proof of \cite[V.3.34.~Lemma]{nakayama} 
for the reader's convenience. 

\begin{rem}\label{f-rem3.3}
In the proof of \cite[V.3.34.~Lemma]{nakayama}, 
$P$ is nef and big in our setting. 
By taking more blow-ups and perturbing the coefficients of 
$\Delta$ slightly, we may further assume that $P$ is 
ample. Therefore, it is easy to see that 
$f_*\mathcal O_X(K_X+\lceil P\rceil)$ is 
an $\omega$-big $\omega$-sheaf. 
For the definition of {\em{$\omega$-big $\omega$-sheaves}}, 
see \cite[V.3.16.~Definition (1)]{nakayama}.  
\end{rem}

By the proof of \cite[V.3.35.~Theorem]{nakayama}, 
we can check the following theorem. It is an application of 
Lemma \ref{f-lem3.1}. 

\begin{thm}[{cf.~\cite[V.3.35.~Theorem]{nakayama}}]\label{f-thm3.4}
Let $f:X\to Y$ be a surjective morphism 
from a normal projective variety $X$ onto a smooth projective 
variety $Y$ with the following properties: 
\begin{itemize}
\item[(i)] $f$ has connected fibers. 
\item[(ii)] $f:(U_X\subset X)\to (U_Y\subset Y)$ is toroidal 
and is equidimensional. 
\item[(iii)] $f$ is smooth over $U_Y$. 
\item[(iv)] $X$ has only quotient singularities. 
\item[(v)] $\Delta_Y=Y\setminus U_Y$. 
\item[(vi)] $\Delta_X$ is a reduced divisor contained in $X\setminus U_X$. 
\item[(vii)] $\Supp f^*\Delta_Y\subset \Supp \Delta_X$. 
\end{itemize}
Let $L$ be a Cartier divisor on $X$ and let $k$ be a positive 
integer with $k\geq 2$ such that $k(K_X+\Delta_X)$ is Cartier.  
Assume that $$L-k(K_{X/Y}+\Delta_X-f^*\Delta_Y)$$ is very ample. 
Then 
$$
\omega_Y(\Delta_Y)\otimes f_*\mathcal O_X(L)
$$ 
is an 
$\widehat \omega$-sheaf. 
\end{thm}

\begin{rem}
A key point of Theorem \ref{f-thm3.4} is that $\Delta_Y$ does not depend on 
$L$.  
\end{rem}

\begin{rem}
We note that $f_*\mathcal O_X(L)$ in Theorem \ref{f-thm3.4} is 
reflexive. This is because $\mathcal O_X(L)$ is 
a locally free sheaf on a normal variety 
$X$ and $f$ is equidimensional. 
For the details, see, for example, \cite[Corollary 
1.7]{hartshorne}. 
We also note that $f$ is flat because 
$f$ is equidimensional, 
$X$ is Cohen--Macaulay, and $Y$ is smooth.
\end{rem}

\begin{cor}\label{f-cor3.7}
In Theorem \ref{f-thm3.4}, there is an ample Cartier divisor $A'$ on $Y$ such that 
$\mathcal O_Y(A')\otimes f_*\mathcal O_X(L)$ is generically generated by 
global sections. Moreover $A'$ is independent of $L$ and 
depends only on $Y$ and $\Delta_Y$. 
\end{cor}

\begin{proof}
Let $A$ be an ample Cartier divisor on $Y$ as in Corollary \ref{f-cor2.19}. 
Then $\mathcal O_Y(A)\otimes \omega_Y(\Delta_Y)\otimes f_*\mathcal O_X(L)$ is 
generically generated by global sections. 
Let $A_1$ be an ample Cartier divisor on $Y$ such that 
$A_1-K_Y-\Delta_Y$ is very ample. Then 
$A'=A+A_1$ is the desired ample Cartier divisor on $Y$. 
Note that $f_*\mathcal O_X(L)$ is reflexive. 
\end{proof}

Let us prove Theorem \ref{f-thm3.4}. 

\begin{proof}[Proof of Theorem \ref{f-thm3.4}]
We take an ample Cartier divisor $H$ on $Y$ such that 
$H=A_1-A_2$, where $A_1$ and $A_2$ are both smooth general very ample 
divisors on $Y$. 
Let $\tau:Y'\to Y$ be a finite Kawamata cover 
from a smooth projective variety $Y'$ such that 
$\tau^*H=mH'$ for some Cartier divisor 
$H'$ on $Y'$ with $m\gg 0$. 
We put 
$$
\xymatrix{
X'\ar[r]^{p}\ar[dr]_{f'}&\widetilde X \ar[r]^{q} \ar[d]^{\widetilde f} & X\ar[d]^{f} \\
  & Y' \ar[r]_{\tau} & Y
} 
$$ 
and $\lambda=q\circ p$, where $\widetilde X=X\times _YY'$ and 
$X'$ is the normalization of $\widetilde X$. 
We may assume that $f':X'\to Y'$ is a weak semistable reduction 
by \cite[Proposition 5.1 and Proposition 5.10]{abramovich-karu}. 
We put $L'=\lambda^*L$. 
Since $L-k(K_{X/Y}+\Delta_X-f^*\Delta_Y)$ is very ample, 
we may assume that $$L=k(K_{X/Y}+\Delta_X-f^*\Delta_Y)+B$$ 
where $B$ is a general smooth very ample divisor on $X$.  
Then, by the 
arguments for the proof of 
\cite[Lemma 10.4 and Lemma 10.5]{fujino}, 
there exists a generically isomorphic injection 
\begin{equation}\label{eq3.1}
f'_*\mathcal O_{X'}(L')\hookrightarrow \tau^*(f_*\mathcal O_X(L)\otimes 
\mathcal O_Y(\Delta_Y)). 
\end{equation}
We put $\Delta=\Delta_X-f^*\Delta_Y$ and define $\Delta'$ by 
$$
k(K_{X'/Y'}+\Delta')=\lambda^*k(K_{X/Y}+\Delta). 
$$
Since $\tau:Y'\to Y$ is a finite 
Kawamata cover, 
we can write $\Delta=\Sigma_X-f^*\Sigma_Y$ such that 
$\Sigma_Y$ is a simple normal crossing 
divisor on $Y$, $\Delta_Y\leq \Sigma_Y$, and $\tau$ is 
\'etale over $Y\setminus \Sigma_Y$. 
Then we have $K_{X'}+\Sigma_{X'}=\lambda ^* 
(K_X+\Sigma_X)$ and 
$K_{Y'}+\Sigma_{Y'}=\tau^*(K_Y+\Sigma_Y)$ such that 
$\Sigma_{X'}$ and $\Sigma_{Y'}$ are effective 
and reduced. 
Of course, $\Delta'=\Sigma_{X'}-f'^*\Sigma_{Y'}$. 
By construction, $\Supp \Sigma_{X'}\supset \Supp 
f'^*\Sigma_{Y'}$. 
Since $f'$ is weakly semistable, 
$f'^*\Sigma_{Y'}$ is reduced. 
Therefore, $\Sigma_{X'}\geq f'^*\Sigma_{Y'}$. 
This means that $\Delta'=\Sigma_{X'}-f'^*
\Sigma_{Y'}$ is effective. 
We can find a positive rational number $\alpha$ 
such that $$L-k(K_{X/Y}+\Delta)-
\alpha f^*H$$ is ample. 
Let $\tau:Y'\to Y$ be the finite Kawamata cover as above for 
$m>(k-1)/\alpha$ and let $H'$ be the same ample divisor as above. 
Then 
$$
L'-k(K_{X'/Y'}+\Delta')-(k-1)f'^*H'=\lambda^*\left(L-
k(K_{X/Y}+\Delta)-\frac{k-1}{m}
f^*H\right)
$$ 
is ample. 
We apply Lemma \ref{f-lem3.1} to 
$L'-(k-1)f'^*H'$ (see also Remark \ref{f-rem3.2}).  
Thus $\omega_{Y'}\otimes f'_*\mathcal O_{X'}(L')$ is an $\widehat \omega$-sheaf. 
Let $G$ be the Galois group of $\tau:Y'\to Y$. 
By the proof of Lemma \ref{f-lem3.1} (see the proof of 
\cite[V.3.34.~Lemma]{nakayama}), 
we can make everything $G$-equivariant and 
have an $\omega$-sheaf $\mathcal F'$ and a generically 
isomorphic $G$-equivariant injection 
$$
\mathcal F'\hookrightarrow \omega_{Y'}\otimes f'_*\mathcal O_{X'}(L'). 
$$ 
Hence there is a generically isomorphic injection 
$$
\mathcal F\hookrightarrow \omega_Y
\otimes f_*\mathcal O_X(L)\otimes \mathcal O_Y(\Delta_Y)
$$ 
from a direct summand $\mathcal F$ of $\tau_*\mathcal F'$. 
Therefore, 
$\omega_Y(\Delta_Y)\otimes f_*\mathcal O_X(L)$ is an $\widehat \omega$-sheaf. 
\end{proof}

Let us prove Theorem \ref{f-thm1.3}. 

\begin{proof}[Proof of Theorem \ref{f-thm1.3}]
Without loss of 
generality, we may assume that $\kappa_\sigma(F, K_F+D_X|_F)\ne -\infty$. 
By \cite[Theorem 2.1, Proposition 4.4, and 
Remark 4.5]{abramovich-karu}, we 
may assume that $f:X\to Y$ satisfies 
the conditions (i)--(v) in Theorem \ref{f-thm3.4}. 
We may also assume that $D_X\subset X\setminus U_X$ and 
$D_Y\subset Y\setminus U_Y$. 
We take $\Delta_X=\Supp (D_X+f^*\Delta_Y)$. 
Then $\Delta_X$ satisfies the conditions (vi) and 
(vii) in Theorem \ref{f-thm3.4}. 
We put 
$$
P=k(K_{X/Y}+\Delta_X-f^*\Delta_Y)
$$ 
and 
$$
D=k(K_{X/Y}+D_X-f^*D_Y)
$$ 
where $k$ is a positive integer $\geq 2$ such that 
$D$ and $P$ are both Cartier. 
We take a very ample Cartier divisor $A$ on $X$. 
We put 
$$
r(mD; A)=\rank f_*\mathcal O_X(mD+A). 
$$ 
Since $D=P$ over the generic point of $Y$, 
$$
r(mD; A)=\rank f_*\mathcal O_X(mP+A). 
$$ 
Note that 
$$
\sigma (D|_F; A|_F)=\max \left\{
k\in \mathbb Z_{\geq 0} \cup \{-\infty\}\, 
\left|\, \underset{m\to \infty}{\limsup}\frac{r(mD; A)}{m^k}>0\right.\right\} 
$$
for a sufficiently general fiber $F$ of $f:X\to Y$. 
We also note that 
\begin{align*}
\kappa _\sigma (F, K_F+D_X|_F)
&=\kappa_\sigma(F, D|_F)\\&=\max\{\sigma(D|_F; A|_F)\, |\, {\text{$A$ is very ample}}\}. 
\end{align*}
Since $f_*\mathcal O_X(mP+A)\otimes \omega_Y(\Delta_Y)$ is a reflexive 
$\widehat{\omega}$-sheaf for every 
positive integer $m$ by Theorem \ref{f-thm3.4}, 
there is an ample Cartier divisor $H$ on $Y$ such that 
we have a generically isomorphic injection 
$$
\mathcal O_Y^{\oplus r(mD; A)}\hookrightarrow 
\mathcal O_Y(H)\otimes f_*\mathcal O_X(mP+A)
$$ 
for every $m\geq 1$ (see Corollary \ref{f-cor3.7}). 
Therefore, we have generically isomorphic injections 
\begin{align*}
&\mathcal O_Y(mk(K_Y+D_Y)+H)^{\oplus r(mD; A)}\\
&\hookrightarrow \mathcal O_Y(mk(K_Y+D_Y)+2H)\otimes 
f_*\mathcal O_X(mP+A)\\
&\hookrightarrow 
\mathcal O_Y(mk(K_Y+D_Y)+2H)\otimes 
f_*\mathcal O_X(mD+A). 
\end{align*}
This implies that 
\begin{align*}
&\dim H^0(X, \mathcal O_X(mk(K_X+D_X)+A+2f^*H))\\
&\geq r(mD; A)\cdot \dim H^0(Y, \mathcal O_Y(mk(K_Y+D_Y)+H)). 
\end{align*} 
We assume that $H$ is sufficiently ample and 
that $A$ is also sufficiently ample. 
Then, by Lemma \ref{f-lem2.8}, 
we can find a constant $C$ such that 
\begin{align*}
&r(mD; A)\cdot \dim H^0(Y, \mathcal O_Y(mk(K_Y+D_Y)+H))
\\ 
&\geq Cm^{\kappa_{\sigma}(F, D|_F)+\kappa _{\sigma}(Y, K_Y+D_Y)}
\end{align*}
for every sufficiently large $m$.  
Hence we have 
$$
\kappa_\sigma(X, K_X+D_X)\geq \kappa_\sigma(F, K_F+D_X|_F)
+\kappa_\sigma(Y, K_Y+D_Y). 
$$
This is the desired inequality. 
\end{proof}

We give a remark on Nakayama's proof of \cite[V.4.1.~Theorem]
{nakayama} for the reader's convenience. 

\begin{rem}\label{f-rem3.8}
The proof of \cite[V.4.1.~Theorem]{nakayama} 
is insufficient. 
We think that we need the inequality 
as in \cite[Theorem 6.7 (7)]{lehmann} 
(see \cite[2.8, 2.10, and Theorem 0.2]{eckl} 
and Lemma \ref{f-lem2.8}). 

Let $D$ be a pseudo-effective $\mathbb R$-divisor on a 
smooth projective variety $X$. 
Then the inequality in \cite[Theorem 6.7 (7)]{lehmann} 
says that 
$$
\kappa_{\sigma} (X, D)=\lim_{m\to \infty} 
\frac{\log \dim H^0(X, \mathcal O_X(\lfloor mD\rfloor +A))}
{\log m}
$$
where $A$ is a sufficiently ample Cartier divisor on $X$. 
This useful characterization is not in \cite{nakayama}. 

From now on, 
we freely use the notation in the proof of \cite[V.4.1.~Theorem]
{nakayama}. Nakayama proved the following inequality 
\begin{equation}\label{f-eq3.2}
h^0(X, \lceil m(D+f^*Q)\rceil +A+2f^*H)
\geq r(mD; A) \cdot h^0(Y, \lfloor mQ\rfloor +H)
\end{equation}
in the proof of \cite[V.4.1.~Theorem (1)]{nakayama}. 
We think that we need \cite[Theorem 6.7 (7)]{lehmann} 
(see also \cite[2.8, 2.10, and Theorem 0.2]{eckl}), which can not 
directly follow from the results in \cite{nakayama}, 
to obtain 
$$
\kappa_{\sigma}(D+f^*Q)\geq \kappa _\sigma(D; X/Y)
+\kappa_{\sigma}(Q)
$$ 
from the inequality \eqref{f-eq3.2}. 
The same trouble is in the proof of \cite[V.4.1.~Theorem (2)]
{nakayama}. 
\end{rem}

We close this section with a sketch of the 
proof of Theorem \ref{f-thm1.9}. 
We leave the details as an 
exercise for the reader. 

\begin{proof}[Sketch of the proof of Theorem \ref{f-thm1.9}] 
Here, we will 
only explain how to modify the proof of Theorem \ref{f-thm1.3} 
for Theorem \ref{f-thm1.9}. 
First, we note that we can easily check that Theorem \ref{f-thm3.4} 
holds true even when the coefficients of the horizontal part 
of $\Delta_X$ are in $[0, 1]\cap \mathbb Q$. 
All we have to do is to check the 
generically isomorphic injection \eqref{eq3.1}
$$
f'_*\mathcal O_{X'}(L')\hookrightarrow 
\tau^*(f_*\mathcal O_X(L)\otimes 
\mathcal O_Y(\Delta_Y)). 
$$ exists when the horizontal part of 
$\Delta_X$ is not necessarily reduced 
in the proof of Theorem \ref{f-thm3.4} 
(see the arguments for the proof of 
\cite[Lemma 10.4 and Lemma 10.5]{fujino}). 
Next, 
by \cite[Theorem 2.1, Proposition 4.4, and 
Remark 4.5]{abramovich-karu}, 
we may assume that $f:X\to Y$ satisfies the conditions 
(i)--(v) in Theorem \ref{f-thm3.4}. 
For the proof of Theorem \ref{f-thm1.9}, 
we may further assume that the coefficients of the 
vertical part of $D_X$ are one by replacing 
$D_X$ with $D^\mathrm{h}_X+\lfloor D^\mathrm{v}_X\rfloor$, 
where 
$D^\mathrm{h}_X$ (resp.~$D^\mathrm{v}_X$) 
is the horizontal (resp.~vertical) 
part of $D_X$. Then we put $\Delta_X=D^\mathrm{h}_X
+\Supp f^*\Delta_Y$. 
Finally, the proof of Theorem \ref{f-thm1.3} works for 
Theorem \ref{f-thm1.9} by the generalization of 
Theorem \ref{f-thm3.4} discussed above. 
\end{proof}

We strongly recommend the interested reader to 
see \cite[V.~\S4]{nakayama} for various 
related results. 

\section{Minimal model program for affine varieties}\label{f-sec4}

In this section, we discuss the minimal model program for affine varieties and 
prove Corollary \ref{f-cor1.6} as an application. 

Let us start with Yoshinori Gongyo's observation. 
Proposition \ref{f-prop4.1} says that the minimal model program works well for 
affine varieties. 

\begin{prop}[Yoshinori Gongyo]\label{f-prop4.1}
Let $V$ be an affine variety. 
We can take a pair $(\overline W, \overline D)$ as in 
Definition \ref{f-def2.4} such that 
$$
\kappa_{\sigma}(\overline W, K_{\overline W}+\overline D)=\kappa 
(\overline W, K_{\overline W}+\overline D)=\overline \kappa (V). 
$$
\end{prop}
\begin{proof}
We take an embedding $V\subset \mathbb A^N$. 
Let $\overline V$ be the closure of $V$ in $\mathbb P^N$. 
Then there is an effective ample Cartier divisor 
$H$ on $\overline V$ such that 
$\Supp H=\overline V\setminus V$. 
We take a resolution $\mu:\overline W\to \overline V$ as in Definition \ref{f-def2.4}. 
Then $\mu^*H$ is an effective Cartier divisor such that 
$\Supp \mu^*H=\Supp \overline D$. 
Let $\varepsilon$ be a small positive rational number 
such that $\overline D-\varepsilon \mu^*H$ is effective. 
Since $\mu^*H$ is semi-ample, 
we can take an effective 
$\mathbb Q$-divisor $B$ on $\overline W$ such that 
$B\sim _{\mathbb Q}\varepsilon \mu^*H$ and 
that $(\overline W, (\overline D-\varepsilon \mu^*H)+B)$ is klt. 
Note that 
$$
K_{\overline W}+\overline D\sim _{\mathbb Q}K_{\overline W}+(\overline D
-\varepsilon \mu^*H)+B
$$ 
and that $(\overline D-\varepsilon \mu^*H)+B$ is big. 
By \cite[Theorem 1.1, Corollary 1.3.3, Corollary 3.9.2]{bchm}, 
$(\overline W, \overline D)$ has a good minimal model or a Mori fiber 
space structure. 
Hence, we obtain 
$
\kappa_{\sigma}(\overline W, K_{\overline W}+\overline D)=\kappa 
(\overline W, K_{\overline W}+\overline D)=\overline \kappa (V)$.  
More precisely, by running a minimal model 
program with ample scaling, 
we have a finite sequence of flips and divisorial contractions 
$$
(\overline W, \overline D)=(\overline W_0, \overline D_0)
\dashrightarrow (\overline W_1, \overline D_1)
\dashrightarrow 
\cdots \dashrightarrow (\overline W_k, \overline D_k)
$$
such that $(\overline W_k, \overline D_k)$ is a good minimal model 
or has a Mori fiber space structure. 
Therefore, $\kappa (\overline W_k, K_{\overline W_k}+\overline D_k)=
\kappa_\sigma (\overline W_k, K_{\overline W_k}+\overline D_k)$ holds. 
Note that in each step of the minimal model program $\kappa$ and 
$\kappa_\sigma$ are preserved. 
Thus, we obtain $\kappa_\sigma(\overline W, K_{\overline W}+\overline D)
=\kappa (\overline W, K_{\overline W}+\overline D)$. 
\end{proof}

\begin{rem}[Logarithmic canonical ring]\label{f-rem4.2}
Let $V$ be an affine variety and let $(\overline W, \overline D)$ be a 
pair as in Definition \ref{f-def2.4}. 
We put 
$$\overline R(V)=
\bigoplus _{m\geq 0}H^0(\overline W, \mathcal O_{\overline W}
(m(K_{\overline W}+\overline D)))
$$
and call it the {\em{logarithmic canonical ring}} of $V$. 
It is well-known and is easy to see that $\overline R(V)$ is independent of 
the pair $(\overline W, \overline D)$ and is well-defined. 
Then $\overline R(V)$ is a finitely generated $\mathbb C$-algebra. 
This is because 
we can choose $(\overline W, \overline D)$ such that 
it has a good minimal model or a Mori fiber space 
structure as we saw in the proof of Proposition \ref{f-prop4.1}.  
\end{rem}

Note that Conjecture \ref{f-conj1.4} follows from the minimal model program and 
the abundance conjecture. 

\begin{prop}\label{f-prop4.3}
Let $X$ be a smooth 
projective variety and let $D_X$ be a simple 
normal crossing 
divisor on $X$. 
Assume that the minimal model program and 
the abundance conjecture hold for $(X, D_X)$. 
Then 
we have 
$$
\kappa (X, K_X+D_X)=\kappa _\sigma(X, K_X+D_X). 
$$
In particular, if $\dim X\leq 3$, then we have 
$$
\kappa (X, K_X+D_X)=\kappa _\sigma(X, K_X+D_X). 
$$
\end{prop}
\begin{proof}
We run the minimal model program. 
If $K_X+D_X$ is pseudo-effective, then $(X, D_X)$ 
has a good minimal model. 
If $K_X+D_X$ is not pseudo-effective, 
then $(X, D_X)$ has a Mori fiber space structure. 
Anyway, we obtain $\kappa (X, K_X+D_X)=\kappa _\sigma
(X, K_X+D_X)$ (see also the proof of Proposition \ref{f-prop4.1}). 
Note that in each step of the minimal model program 
$\kappa$ and $\kappa _\sigma$ are preserved. 
\end{proof}

\begin{proof}[Proof of Corollary \ref{f-cor1.5}] 
This is obvious by Theorem \ref{f-thm1.3} and Proposition \ref{f-prop4.3}. 
\end{proof}

Let us prove Corollary \ref{f-cor1.6}. 

\begin{proof}[Proof of Corollary \ref{f-cor1.6}]
We take the following commutative diagram: 
$$
\xymatrix{
V \ar@{^{(}->}[r] \ar[d]_g & V'\ar[d]_{h}&V'\ar[d]\ar@{=}[l]&X\ar[d]^{f}\ar[l]_\alpha \\
W \ar@{^{(}->}[r] &W'&W''\ar[l]&Y\ar[l]_{\beta}
}
$$
such that $h:V'\to W'$ is a compactification of $g:V\to W$, 
$V'\to W''\to W'$ is the Stein factorization of $V'\to W'$, 
$\alpha$ and $\beta$ are 
suitable resolutions. 
We can take a simple normal crossing divisor $D_X$ on $X$ such that 
$$\overline \kappa (V)=\kappa(X, K_X+D_X)=\kappa_\sigma (X, K_X+D_X)$$ 
by Proposition \ref{f-prop4.1}. 
We have a simple normal crossing divisor $D_Y$ on $Y$ such that 
$\Supp f^*D_Y\subset \Supp D_X$ and 
$$\overline \kappa (W)\leq 
\kappa (Y, K_Y+D_Y)\leq \kappa _{\sigma}(Y, K_Y+D_Y) $$ 
by Lemma \ref{f-lem2.14}. 
Then, by Theorem \ref{f-thm1.3}, 
we obtain 
\begin{align*}
\overline{\kappa}(V)=\kappa(X, K_X+D_X)&=\kappa_{\sigma}(X, K_X+D_X)\\&\geq 
\kappa _{\sigma}(F, K_F+D_X|_F)+\kappa_{\sigma}(Y, K_Y+D_Y)\\
&\geq 
\kappa (F, K_F+D_X|_F)+\kappa(Y, K_Y+D_Y)\\
&\geq \overline{\kappa}(F')+\overline{\kappa}(W)
\end{align*} 
where $F$ is a sufficiently general fiber of $f:X\to Y$. 
Note that $$\overline{\kappa}(F')=\kappa(F, K_F+D_X|_F). $$ 
Therefore, we obtain the desired inequality of the logarithmic Kodaira dimension. 
\end{proof}

\begin{rem}\label{f-rem4.4} 
If $g:V \dashrightarrow W$ is a strictly rational dominant 
map, then we can take a proper birational morphism 
$\mu:\widetilde V\to V$ such that 
$g\circ \mu:\widetilde V\to W$ is a morphism. 
By applying the proof of Corollary \ref{f-cor1.6} to 
$g\circ \mu:\widetilde V\to W$, we have 
$\overline \kappa (V)\geq \overline \kappa (F')+\overline 
\kappa (W)$ as pointed out in Remark \ref{f-rem1.7}. 
\end{rem}

We close this paper with a remark on Conjecture 
\ref{f-conj2.5} (see Conjecture \ref{f-conj1.1}). 

\begin{rem}\label{f-rem4.5}
By the proof of Corollary \ref{f-cor1.6}, 
we see that Conjecture \ref{f-conj2.5} 
(see Conjecture \ref{f-conj1.1}) 
follows from $\kappa (X, K_X+D_X)=\kappa _\sigma(X, K_X+D_X)$. 
Moreover, the equality 
$\kappa (X, K_X+D_X)=\kappa _\sigma(X, K_X+D_X)$ 
follows from 
the minimal model program and the abundance conjecture 
for $(X, D_X)$ by Proposition \ref{f-prop4.3}. 
Therefore, Conjecture \ref{f-conj2.5} (see Conjecture \ref{f-conj1.1}) 
now becomes a consequence of the minimal 
model program and the abundance conjecture by Theorem \ref{f-thm1.3}. 
This fact strongly supports Conjecture \ref{f-conj2.5} (see 
Conjecture \ref{f-conj1.1} and Conjecture \ref{f-conj1.2}). 
\end{rem}


\end{document}